\documentclass[oneside]{amsart}

\usepackage[letterpaper,body={15.0cm,22.5cm}, mag=1000]{geometry}
\usepackage{amssymb}
\usepackage{amsthm}
\usepackage{amscd}

\theoremstyle{plain}
\newtheorem*{Thma}{Theorem A}
\newtheorem*{Thmb}{Theorem B}
\newtheorem{Thm}{Theorem}[section]
\newtheorem{Lem}[Thm]{Lemma}
\newtheorem{Cor}[Thm]{Corollary}
\theoremstyle{definition}
\newtheorem{Rem}[Thm]{Remark}

\DeclareMathOperator{\cs}{cs}
\DeclareMathOperator{\cd}{cd}
\DeclareMathOperator{\Irr}{Irr}
\DeclareMathOperator{\Cl}{C\ell} 
\DeclareMathOperator{\crk}{crk}

\begin{document}

\baselineskip 15pt

\title[ Some bounds on commutativity degree]
	{Some bounds on relative commutativity degree}

\author[R. K. Nath and M. K. Yadav ]
	{Rajat K. Nath and Manoj K. Yadav}

\address{
School of Mathematics\\
Harish-Chandra Research Institute\\
Chhatnag Road, Jhunsi\\
Allahabad-211 019, U. P., India.} 
\email{rajatkantinath@yahoo.com and myadav@hri.res.in}

\begin{abstract}
 The relative commutativity degree of a subgroup $H$ of a finite group $G$, denoted by $\Pr(H, G)$, is the probability that an element of $G$ commutes with an element of $H$. In this article we obtain some lower and upper bounds for $\Pr(H, G)$ and their consequences. We also study an invariance property of $\Pr(H, G)$ and its generalizations, under isoclinism of pairs of groups.
\end{abstract}

\subjclass[2000]{Primary 20D60; Secondry 20P05}
\keywords{commutativity degree, conjugacy class size, isoclinic pairs of groups}
 
\maketitle
\section{Introduction} \label{S:intro}
Let $H$  be a subgroup of a finite group $G$. Let $Z(G)$ denote the center of $G$ and define $Z(H, G) := \{h \in H : hg = gh \, \forall \, g \in G\}  = Z(G) \cap H$. The commutator $xyx^{-1}y^{-1}$ of any two elements $x, y \in G$ is denoted by $[x, y]$. By $K(G, H)$, we denote the set $\{[x, y] : x \in G, y \in H\}$. The subgroup generated by $K(G, H)$ is denoted by $[G, H]$. Notice that $[G, G]$ is the commutator subgroup of $G$, which is also denoted by $G'$. For any element $x \in G$, by $\Cl_G(x)$ and $C_G(x)$ we denote the conjugacy class of $x$ in $G$ and centralizer of $x$ in $G$ respectively. By $K(G)$ we denote $K(G, G)$.  The relative commutativity degree of  $H$, denoted by $\Pr(H, G)$, is the probability that an element of $G$ commutes with an element of $H$. This notion has been introduced and studied in \cite{ERL07}. It is clear that if $H = G$, then $\Pr(H, G)$ coincides with $\Pr(G)$ which is known as the commutativity degree or commuting probability of $G$ (see \cite{GR06, wG73, L95, ND10, RE09}). Also, $\Pr(H, G) = 1$ if and only if $Z(H, G) = H$.

Importance of studying lower bound for $\Pr(G)$ goes back to 1973, where W. H. Gustafson \cite{wG73} emphasized on getting this bound for an arbitrary finite group. Such a bound has been studied, under certain conditions on the group, in \cite[Theorem 3.5]{ERL07} and
\cite[Corollary 2.3]{mP04}. We obtain a better lower bound for $\Pr(G)$ as a corollary (Corollary \ref{cor:Thm1}) of the following theorem, which we prove in Section 3.

\begin{Thma}\label{Thma}
Let $H$ be  a  subgroup of a finite group $G$. Then
\[
\Pr(H, G) \geq \frac{1}{|K(G, H)|}\left(1 + \frac{|K(G, H)| - 1}{|H:Z(H, G)|}\right).
\]
In particular, if $Z(H, G) \ne H$ then $\Pr(H, G) > \frac{1}{|K(G, H)|}$. 
\end{Thma}

Following \cite{DN10}, for $g \in G$, we define ${\Pr}_g(H, G)$ to be $\frac{|\{(x, y) \in H \times G : xyx^{-1}y^{-1} = g\}|}{|H||G|}$.
Notice that if $H = G$, then ${\Pr}_g(G) := {\Pr}_g(G, G)$ measures the probability that the commutator of two group elements is equal to a given element, which is studied extensively in \cite{mP04} by M. R. Pournaki and R. Sobhani.  In \cite[Proposition 3.1]{mP04}, ${\Pr}_g(G)$ is computed for groups $G$ with  $|G'| = p$, a prime (not necessarily the smallest one), and $G' \leq Z(G)$, using character theory of finite groups. In the following theorem, which we prove in Section 3, we generalize this result and prove it without using character theory.

\begin{Thmb} \label{Thm3.2}
Let $H$ be a  subgroup of a finite nilpotent group $G$. If $|[G, H]| = p$, a prime (not necessarily the smallest one), and $g \in [G, H]$ then 
\[
{\Pr}_g(H, G) = \begin{cases} \frac{1}{p}\left(1 + \frac{p - 1}{|H : Z(H, G)|}\right) \text{ if } g = 1\\
\frac{1}{p}\left(1 - \frac{1}{|H : Z(H, G)|}\right) \text{ if } g \ne 1.
\end{cases}
\]
\end{Thmb}

In Section 2, following \cite{SSK07}, we define isoclinism on the class of all pairs of groups $(G, H)$, where $H$ is any subgroup of $G$. Then we show that $\Pr_g(H, G)$ is invariant under isoclinism of pairs of groups.  
In Section 4, we obtain lower and upper bounds for $\Pr(H, G)$ in terms of largest and smallest class sizes of elements of $H - Z(H, G)$ repectively.  In the last section we observe that for a finite group $G$ with given information on the conjugacy class sizes,  the information about $\Pr(G)$ can be used to say something regarding the solvability and supersolvability of $G$.

\section{Isoclinic pairs of groups}

In 1940, P. Hall \cite{pH40} introduced the following concept of isoclinism on the class of all groups:

Let $X$ be a finite group and $\bar{X} = X/Z(X)$. 
Then commutation in $X$ gives a well defined map
$a_{X} : \bar{X} \times \bar{X} \to X'$ such that
$a_{X}(xZ(X), yZ(X)) = [x,y]$ for $(x,y) \in X \times X$.
Two finite groups $G$ and $H$ are called \emph{isoclinic} if 
there exists an  isomorphism $\alpha$ of the factor group
$\bar G = G/Z(G)$ onto $\bar{H} = H/Z(H)$, and an isomorphism $\beta$ of
the subgroup $G'$ onto  $H'$
such that the following diagram is commutative
\[
 \begin{CD}
   \bar G \times \bar G  @>a_G>> G'\\
   @V{\alpha\times\alpha}VV        @VV{\beta}V\\
   \bar H \times \bar H @>a_H>> H'.
  \end{CD}
\]
The resulting pair $(\alpha, \beta)$ is called an \emph{isoclinism} of $G$ 
onto $H$. Notice that isoclinism is an equivalence relation among finite 
groups.

In 2008, M. R. Pournaki and R. Sobhani \cite{mP04} proved that if $(\alpha, \beta)$ is an isoclinism of $G$ onto $H$ and $g \in G'$, then ${\Pr}_g(G) = \Pr_{\beta(g)}(H)$. Recently, Salemkar et. al \cite{SSK07}  have introduced the concept of isoclinism on the class of all pairs of groups $(G, H)$, where $H$ is a normal subgroup of $G$. Notice that isoclinism on pairs of groups $(G, H)$ can be defined for  any subgroup $H$ of $G$ in the following way:

Let $X$ be a finite group and $Y$ be a subgroup of $X$. Let $\bar{Y} = Y/Z(Y, X)$ and $\bar{X} = X/Z(Y, X)$. Then the map $a_{(Y, X)}: \bar{Y} \times \bar{X} \to [Y, X]$ defined by 
$a_{(Y, X)}(yZ(Y, X), xZ(Y, X)) = [y, x]$, is  well defined. 

Let $G_1$ and $G_2$ be two groups with subgroups $H_1$ and $H_2$ respectively. A pair of groups $(G_1, H_1)$ is said to be isoclinic to a pair of groups $(G_2, H_2)$ if there exists an isomorphism $\alpha$ from $\bar{G_1} = G_1/Z(H_1, G_1)$ onto $\bar{G_2} = G_2/Z(H_2, G_2)$ such that $\alpha (H_1/Z(H_1, G_1)) = H_2/Z(H_2, G_2)$, and an isomorphism $\beta : [H_1, G_1] \to [H_2, G_2]$, such that the following diagram commutes
\begin{equation}\label{diag5}
 \begin{CD}
   \bar{H_1} \times \bar{G_1}  @>a_{(H_1, G_1)} >> [H_1, G_1]\\
   @V{\alpha\times\alpha}VV           @VV{\beta}V\\
   \bar{H_2} \times \bar{G_2} @>a_{(H_2, G_2)} >> [H_2, G_2],
  \end{CD}
\end{equation} 
where $\bar{H_1} = H_1/Z(H_1, G_1)$ and $\bar{H_2} = H_2/Z(H_2, G_2)$.
Notice that isoclinism is an equivalence relation among pairs of finite 
groups.
 
\vspace{.2in}

The following result extends  \cite[Theorem 3.3]{STM10} and  \cite[Lemma 3.5]{mP04}.
\begin{Thm}\label{Thmiso}
Let $(G_1, H_1)$ and $(G_2, H_2)$ be two pairs of groups and $(\alpha, \beta)$ be an isoclinism from $(G_1, H_1)$ to $(G_2, H_2)$. If $g \in [H_1, G_1]$ then $\Pr_g(H_1, G_1) = \Pr_{\beta(g)}(H_2, G_2)$. 
\end{Thm}
\begin{proof}
Let us set $Z_1 = Z(H_1, G_1)$ and $Z_2 = Z(H_2, G_2)$. Since $(\alpha, \beta)$ is an isoclinism from $(G_1, H_1)$ to $(G_2, H_2)$, diagram \eqref{diag5} commutes. Let $g \in [H_1, G_1]$ be the given element. Consider the sets $S_g = \{(h_1Z_1, g_1Z_1) \in H_1/Z_1  \times G_1/Z_1 : [h_1, g_1] = g\}$ and $S_{\beta(g)} = \{(h_2Z_2, g_2Z_2) \in H_2/Z_2  \times G_2/Z_2 : [h_2, g_2] = \beta(g)\}$. Since diagram \eqref{diag5} commutes, it follows that $|S_g| = |S_{\beta(g)}|$. Since the map $a_{(H_i, G_i)}$ for $i = 1, \;2$ is well defined, it follows that  $|\{(h_1, g_1) \in H_1 \times G_1 : [h_1, g_1] = g\}| = |Z_1|^2 |S_g|$ and $|\{(h_2, g_2) \in H_2 \times G_2 : [h_2, g_2] = \beta(g)\}| = |Z_2|^2 |S_{\beta(g)}|$. Also, notice that $|H_1 : Z_1| = |H_2 : Z_2|$ and $|G_1 : Z_1| = |G_2 : Z_2|$. Hence
\[
{\Pr}_g(H_1, G_1) = \frac{|S_g|}{|H_1 : Z_1| |G_1 : Z_1|} = \frac{|S_{\beta(g)}|}{|H_2 : Z_2| |G_2 : Z_2|}  = {\Pr}_{\beta(g)}(H_2, G_2).
\]
This completes the proof. 
\end{proof}

The concept of isoclinism can be generalised in the following way:

Let $H$ be a finite group and $H_1, H_2, \dots, H_{m+1}$ be any $m+1$ subgroups of $H$. Let $Z_m(H_i)$ denote the subgroup $H_i \cap Z_m(H)$ for $1 \le i \le m+1$, where $Z_m(H)$ is the $m$th term in the upper central series of $H$. Set $\times_{i=1}^{m+1}H_i = H_1 \times H_2 \times \cdots \times H_{m+1}$.
It follows from \cite[IV, 7.6]{BT82} that the map $a_{(H_1, H_2, \dots, H_{m+1})}$ from 
$\times_{i=1}^{m+1}\frac{H_i}{Z_m(H_i)}$ onto $[H_1, H_2, \dots, H_{m+1}] = [\cdots[ H_1, H_2], \dots, H_{m+1}]$ defined by 
\[a_{(H_1, H_2, \dots, H_{m+1})}(h_1Z_m(H_1), h_2Z_m(H_2), \dots, h_{m+1}Z_m(H_{m+1})) = [h_1, h_2, \dots, h_{m+1}]\]
is well defined.

Let $K$ another group and $K_1, K_2, \dots, K_{m+1}$ be any $m+1$ subgroups of $K$. Again set $Z_m(K_i) = K_i \cap Z_m(K)$. Then $(H_1, H_2, \dots, H_{m+1})$ is said to be $(m+1)$-tuple isoclinic to $(K_1, K_2, \dots, K_{m+1})$ if there exist an isomorphism $\alpha$ from $\frac{H}{Z_m(H)}$ onto $\frac{K}{Z_m(K)}$ such that $\alpha_i(\frac{H_i}{Z_m(H_i)}) = \frac{K_i}{Z_m(K_i)}$, where $\alpha_i = \alpha|_{\frac{H_i}{Z_m(H_i)}}$ for $1 \le i \le m+1$, and an isomorphism $\beta$ from $[H_1, H_2, \dots, H_{m+1}]$ onto $[K_1, K_2, \dots, K_{m+1}]$ such that the following diagram commutes
\begin{equation*}
 \begin{CD}
   \times_{i=1}^{m+1} \frac{H_i}{Z_m(H_i)}  @>a_{(H_1, H_2, \dots, H_{m+1})} >> [H_1, H_2, \dots, H_{m+1}]\\
   @V{\times_{i=1}^{m+1}\alpha_i}VV           @VV{\beta}V\\
   \times_{i=1}^{m+1} \frac{K_i}{Z_m(K_i)} @> a_{(K_1, K_2, \dots, K_{m+1})} >> [K_1, K_2, \dots, K_{m+1}].
  \end{CD}
\end{equation*} 
The pair $(\alpha, \beta)$ is called an $(m+1)$-tuple isoclinism of $(H_1, H_2, \dots, H_{m+1})$ onto $(K_1, K_2, \dots, K_{m+1})$.

Let $g \in [H_1, H_2, \dots, H_{m+1}]$. Then we can define
\[{\Pr}_g(H_1, H_2, \dots, H_{m+1}) = \frac{1}{\prod_{i=1}^{m+1}|H_i|} |\{(h_1, h_2, \dots, h_{m+1}) \in \times_{i=1}^{m+1}{H_i} : [h_1, h_2, \dots, h_{m+1}] = g\}|.\]

\begin{Rem}
Let $(\alpha, \beta)$ be an $(m+1)$-tuple isoclinism of $(H_1, H_2, \dots, H_{m+1})$ onto $(K_1, K_2, \dots,$ $ K_{m+1})$.
 Then using arguments as in the proof of Theorem \ref{Thmiso}, one can prove that 
\[{\Pr}_g(H_1, H_2, \dots, H_{m+1}) =  {\Pr}_{\beta(g)}(K_1, K_2, \dots, K_{m+1}).\]
\end{Rem}

\section{ Lower bounds for $\Pr(H, G)$} \label{bound}

Notice that $yxy^{-1} = yxy^{-1}x^{-1}x \in K(G, H)x$ for all $y \in G$ and  $x \in H$.
So, it follows that 
\begin{equation} \label{eq1}
 \Cl_G(x)  \subseteq  K(G, H)x 
\end{equation}
for all $x \in H$.

The following lemma can be  derived from \cite[Theorem 2.3]{DN10}. However, for completeness, we give a slightly modified proof here.

\begin{Lem}\label{eq2}
Let $H$ be a  subgroup of a finite group $G$ and $g \in G$. Then
\[
 {\Pr}_g(H, G) = \frac{1}{|H|}\underset{g^{-1}x \in \Cl_G(x)}{\sum_{x \in H}}\frac{1}{|\Cl_G(x)|}.
\]
\end{Lem}
\begin{proof}
Notice that $\{(x, y) \in H \times G : [x, y] = g\} = \underset{x \in H}{\cup}(\{x\} \times T_x)$, where $T_x = \{(y \in G : [x, y] = g\}$. Further notice that, for any $x \in H$, the set $T_x$ is non-empty if and only if $g^{-1}x \in \Cl_G(x)$.
Suppose that $T_x$ is non-empty for some $x \in H$. Fix an element $t \in T_x$. It is easy to see that $T_x = tC_G(x)$. Hence
\[
 {\Pr}_g(H, G) = \frac{1}{|H|}\underset{g^{-1}x \in \Cl_G(x)}{\sum_{x \in H}}\frac{1}{|G:C_G(x)|}
\]
and the lemma follows. 
\end{proof}

The following lemma is an easy exercise.

\begin{Lem} \label{newlemma}
 Let $H$ be a subgroup of a finite group $G$. Then
\[
\frac{1}{n}\left(1 + \frac{n - 1}{|H : Z(H, G)|}\right) \geq \frac{1}{m}\left(1 + \frac{m - 1}{|H : Z(H, G)|}\right)
\]
for any two positive integers $m, n$ such that $m \geq n$. If $Z(H, G) \neq H$, then the equality holds if and only if $m = n$.
\end{Lem}

Now we are ready to prove Theorem A.

\noindent {\it Proof of Theorem A.} Let $G$ be a group and $H$ be a subgroup of $G$.
Then, by putting $g = 1$ in Lemma \ref{eq2}, we get
\begin{align*}
\Pr(H, G) & = \frac{1}{|H|}\sum_{x \in H}\frac{1}{|\Cl_G(x)|}\\
& = \frac{1}{|H|}\left(|Z(H, G)| + \sum_{x \in H - Z(H, G)}\frac{1}{|\Cl_G(x)|}\right)\\
& \geq  \frac{1}{|H|}\left(|Z(H, G)| +  \frac{|H| - |Z(H, G)|}{|K(G, H)|}\right), \, \text{using   \eqref{eq1}},\\
& = \frac{1}{|K(G, H)|}\left(1 + \frac{|K(G, H)| - 1}{|H:Z(H, G)|}\right).
\end{align*}
This completes the proof.
\hfill $\Box$

\vspace{.3in}

We would like to remark here that Erfanian et. al \cite[Theorem 3.5]{ERL07} proved 
\[
\Pr(H, G) \geq \frac{|Z(H, G)|}{|H|} + \frac{p(|H| - |Z(H, G)|)}{|H||G|},
\]
where $p$ is the smallest prime dividing $|G|$.
Further, Salemkar et. al \cite[Theorem 2.2 (vi)]{STM10} proved 
\[
\Pr(H, G) \geq \frac{1}{|[G, H]|}\left(1 + \frac{|[G, H]| - 1}{|H:Z(H, G)|}\right)
\] 
for any normal subgroup $H$ of $G$ with $[G, H] \leq Z(H, G)$.
 
 However, if $[G, H] \ne G$ and $H \neq Z(H, G)$, then  it can be checked easily that
\[
\frac{1}{|[G, H]|}\left(1 + \frac{|[G, H]| - 1}{|H:Z(H, G)|}\right) \geq \frac{|Z(H, G)|}{|H|} + \frac{p(|H| - |Z(H, G)|)}{|H||G|}
\]
with equality if and only if $|G : [G, H]| = p$.

It follows from Lemma \ref{newlemma} that
\[
\frac{1}{|K(G, H)|}\left(1 + \frac{|K(G, H)| - 1}{|H:Z(H, G)|}\right) \geq \frac{1}{|[G, H]|}\left(1 + \frac{|[G, H]| - 1}{|H:Z(H, G)|}\right)
\]
and equality holds if and only if $K(G, H) = [G, H]$.
{\it This shows that the lower bound obtained in  Theorem A is better than the known bounds mentioned above}.

Examples of groups $G$ such that $|K(G)| < |G'|$ can be found in \cite{KM07}. If we take $H$ as a maximal subgroup of such a group $G$, then $K(G, H)$ is properly contained in $[G, H] = G'$.

Putting $H = G$ in Theorem A and noticing that $Z(G) = Z(G, G)$, we get the following corollary.

\begin{Cor}\label{cor:Thm1}
 If $G$ is a finite group then
\[
 \Pr(G) \geq \frac{1}{|K(G)|}\left(1 + \frac{|K(G)| - 1}{|G : Z(G)|}\right).
\]
 In particular, if $G$ is non-abelian, then $\Pr(G) > \frac{1}{|K(G)|}$.
\end{Cor}

Pournaki et. al \cite[Corollary 2.3]{mP04} obtained 
\begin{equation}\label{lbd}
 \Pr(G) \geq \frac{1}{|G'|}\left(1 + \frac{|G'| - 1}{|G : Z(G)|}\right)
\end{equation}
if $G$ has only two complex irreducible character degrees. Recently, the first author together with A. K. Das  obtained the same lower bound for $\Pr(G)$ without any restriction on $G$ (see \cite[Theorem 1]{ND10}). {\it However,  using Lemma \ref{newlemma}, it is easy to see that the lower bound obtained in Corollary \ref{cor:Thm1} is better than the bound in equation \eqref{lbd}}.

The following corollary gives  some equivalent necessary and sufficient conditions for equality to hold in Theorem A.

\begin{Cor}\label{Thm2}
Let $H$ be a  subgroup of a finite group $G$. If $Z(H, G) \ne H$, then the following statements are equivalent:
\begin{enumerate}
\item $\Pr(H, G) = \frac{1}{|K(G, H)|}\left(1 + \frac{|K(G, H)| - 1}{|H:Z(H, G)|}\right)$.

\item $\Cl_G(x)  = K(G, H)x$ for all $x \in H - Z(H, G)$.

\item $K(G, H) = \{yxy^{-1}x^{-1} : y \in G\}$ for all $x \in H - Z(H, G)$.
\end{enumerate} 
\end{Cor}

\begin{proof}
It is easy to see that (ii) and (iii) are equivalent.
The equivalence of (i) and (ii) follows from the proof of Theorem A and equation \eqref{eq1}. 
\end{proof}

A finite group $G$ is said to be a Camina group if $\Cl_G(x) = G'x$ for all $x \in G - G'$.
Let $H = G$ in Corollary \ref{Thm2}, then $G$ has only two conjugacy class sizes, namely $1$ and $|K(G)|$. Now it follows from \cite{nI53} that $G$ is a direct product of a $p$-group and an abelian group. It then follows from \cite{Ishikawa02} that the nilpotency class of $G$ is at most $3$. We claim that in our case the nilpotency class is at most $2$. Supose that the nolpotency class is $3$. Then there exists an element $u \in G' - Z(G)$. Then by (iii) $G' = \langle [G, u] \rangle \le [G, G']$, which is a contradiction. Thus the equivalent conditions in this corollary are satisfied if and only if $G$ is isoclinic to some Camina $p$-group of class $2$, for some prime $p$. 

There are also examples in the case when $H \neq G$. 

\noindent {\bf 1.} Let $G$ be any Camina $p$-groups of class $2$ and $H$ be a maximal subgroup of $G$. Then $G' = [G, H] = K(G, H)$, $Z(H, G) = Z(G)$ and for all $x \in H - Z(G)$,  $\Cl_G(x)  = K(G, H)x$. 

\noindent {\bf 2.} Let $G$ be any Camina $p$-groups of class $3$ and $H = G'$. Then $[G, H] = [H, G] = [G', G] = K(G, H)$, $Z(H, G) = Z(G)$ and for all $x \in H - Z(G)$,  $\Cl_G(x)  = K(G, H)x$. 

\noindent {\bf 3.} Our next example is a finite group of order $p^5$, which is not a Camina $p$-group, $p$ is an odd prime. Let
\begin{align*}
G = &\langle a_1, a_2, b, c_1, c_2 : [a_1, a_2] = b, [a_1, b] = c_1, [a_2, b] = c_2,\\
  & a_1^p = a_2^p = b^p = c_1^p = c_2^p = 1\rangle. 
\end{align*}
It follows from \cite{Ishikawa99} that $Z(G) < G'$, $|Z(G)| = p^2$, $|G'| = p^3$ and $|\Cl_G(x)| = p^2$ for all $x \in G - Z(G)$. Now take $H = G'$. Then $[G, H] = [H, G] = [G', G] = K(G, H)$, $Z(H, G) = Z(G)$ and for all $x \in H - Z(G)$,  $\Cl_G(x)  = K(G, H)x$.  

\vspace{.2in}

Now, we compute ${\Pr}_g(H, G)$, as an application of the above results, for some classes of finite groups. Recall that for any two groups $G_1$ and $G_2$ with subgroups $H_1$ and $H_2$ respectively, we have

\begin{equation}\label{prod}
 \Pr(H_1 \times H_2, G_1 \times G_2) = \Pr(H_1, G_1)\Pr(H_2, G_2).
\end{equation}

\begin{Lem} \label{Thm3.1}
Let $H$ be a  subgroup of a finite group $G$ and $p$ be the smallest prime dividing $|G|$. If $|[G, H]| = p$ and $g \in [G, H]$, then 
\[
{\Pr}_g(H, G) = \begin{cases} \frac{1}{p}\left(1 + \frac{p - 1}{|H : Z(H, G)|}\right) \text{ if } g = 1\\
\frac{1}{p}\left(1 - \frac{1}{|H : Z(H, G)|}\right) \text{ if } g \ne 1.
\end{cases}
\]
\end{Lem}
\begin{proof}
Since $|[G, H]| = p$, it follows that $K(G, H) = [G, H]$, and $|\Cl_G(x)| = |[G, H]| = p$ and $\Cl_G(x) =  K(G, H)x$ for all $x \in H - Z(H, G)$. Hence, if $g = 1$ then, by Corollary \ref{Thm2}, we have
\[
{\Pr}_g(H, G) =  \frac{1}{p}\left(1 + \frac{p - 1}{|H : Z(H, G)|}\right). 
\]

Now consider that $g \ne 1$. Since $\Cl_G(x) =  K(G, H)x$ for all $x \in H - Z(H, G)$ and $K(G, H)$ is a subgroup, it follows that $g^{-1}x \in \Cl_G(x)$. Notice that if $x \in Z(H, G)$ and $g^{-1}x \in \Cl_G(x)$, then $g = 1$. Thus by  Lemma \ref{eq2}, we have 
\begin{align*}
{\Pr}_g(H, G) & = \frac{1}{|H|}\sum_{x \in H - Z(H, G)}\frac{1}{|\Cl_G(x)|}\\
& = \frac{1}{p}\left(1 - \frac{1}{|H : Z(H, G)|}\right). 
\end{align*}
This completes the proof.   
\end{proof}

Now we are ready to prove Theorem B.

\noindent {\it Proof of Theorem B.} 
Let $H$ be a subgroup of a finite nilpotent group $G$ and $p$ be a prime integer such that $|[G, H]| = p$. Let $P$ and $Q$ be the Sylow $p$-subgroups of $G$ and $H$. Then there exist Hall $p'$-subgroups $A$ and $B$ of $G$ and $H$ respectively such that $G = P \times A$ and 
$H = Q \times B$. Since $|[G, H]| = p$, it follows that $|[P, Q]| = p$ and $[A, B] = 1$. Therefore $[G, H] = [P, Q]$. Let $g \in [G, H] = [P, Q]$. Notice that $|\{(x, y) \in H \times G : [x, y] = g\}| = |A| |B| |\{(u, v) \in Q \times P : [u, v] = g\}|$. Hence
${\Pr}_g(H, G)  = {\Pr}_g(Q, P)$. Since $P$ is a $p$-group such that $|[P, Q]| = p$, hypothesis of Lemma \ref{Thm3.1} is satisfied. Hence
\[
{\Pr}_g(H, G)  = {\Pr}_g(Q, P) = \begin{cases} \frac{1}{p}\left(1 + \frac{p - 1}{|Q : Z(Q, P)|}\right) \text{ if } g = 1\\
\frac{1}{p}\left(1 - \frac{1}{|Q : Z(Q, P)|}\right) \text{ if } g \ne 1.
\end{cases}
\]
Notice that $|Q : Z(Q, P)| = |H : Z(H,G)|$. This completes the proof of the theorem.
\hfill $\Box$

\vspace{.3in}

We would like to remark that the hypothesis of Theorem B is naturally satisfied in many cases, namely (1) if $G$ is an extraspecial $p$-group and $H$ is any non-central subgroup of $G$; (2) if $G$ is a $p$-group of maximal class of order $p^n$ and $H = \gamma_{n-1}(G)$, where $n$ is a positive integer and $\gamma_{n-1}(G)$ denotes the $(n-1)$th term of the lower central series of $G$. More generally, if we take $G$ to be any finite $p$-group of nilpotency class $c$ such that $|\gamma_c(G)| = p$ and $H = \gamma_{c-1}(G)$, then $|[G, H]| = p$.

\section{Some more bounds for $\Pr(H, G)$}

For a given finite group $G$ and a subgroup $H$ of $G$ such that $Z(H, G) \neq H$, let $\cs(G, H)$ denote $\{|\Cl_G(x)| : x \in H\}$. Clearly, $\cs(G) := \cs(G, G)$ is the set of conjugacy class sizes of $G$. In the following theorem we give some bounds for $\Pr(H, G)$ in terms of the largest and smallest conjugacy class size of the elements of $H - Z(H, G)$. Let $s_H = \min \{|\Cl_G(x)| : x \in H - Z(H, G)\}$ and $l_H = \max \{|\Cl_G(x)| : x \in H - Z(H, G)\}$. 

\begin{Thm}\label{ThmUB}
If $H$ is a  subgroup of a finite group $G$ such that $Z(H, G) \ne H$  then
\[
\frac{1}{l_H}\left( 1 + \frac{l_H - 1}{|H : Z(H, G)|}\right) \leq \Pr(H, G) \leq \frac{1}{s_H}\left( 1 + \frac{s_H - 1}{|H : Z(H, G)|}\right)
\]
with equality if and only if $\cs(G, H) = \{1, s_H = l_H\}$.
\end{Thm}

\begin{proof}
By Lemma \ref{eq2} we have
\begin{align*}
\Pr(H, G) & = \frac{1}{|H|}\sum_{x \in H}\frac{1}{|\Cl_G(x)|}\\
& \leq \frac{1}{|H|}\left( |Z(H, G)| + \frac{|H| - |Z(H, G)|}{s_H}\right)\\
& = \frac{1}{s_H}\left( 1 + \frac{s_H - 1}{|H : Z(H, G)|}\right).
\end{align*}
The equality holds if and only if $\cs(G, H) = \{1, s_H\}$.
Similarly the other bound can be obtained. 
\end{proof}

If $p$ is the smallest prime dividing $|G|$ and $H$ a subgroup of $G$ then it has also been shown in \cite[Theorem 3.5]{ERL07} that
\[
\Pr(H, G) \leq \frac{1}{2}\left(1 + \frac{1}{|H : Z(H, G)|}\right). 
\]
However, Theorem 3.8 of \cite{DN10} gives
\[
\Pr(H, G) \leq \frac{1}{p}\left(1 + \frac{p - 1}{|H : Z(H, G)|}\right). 
\]
This bound is also obtained by Salemkar et. al  for any normal subgroup $H$ of $G$ (see \cite[Theorem 2.2(v)]{STM10}). 

By  Lemma \ref{newlemma}, we have, for any positive integer $n$ such that $2 \leq n \leq s_H$,
\[
\frac{1}{s_H}\left( 1 + \frac{s_H - 1}{|H : Z(H, G)|}\right) \leq \frac{1}{n}\left( 1 + \frac{n - 1}{|H : Z(H, G)|}\right)
\]
with equality if and only if $s_H = n$. 

Let $p$ be the smallest prime dividing $|G|$. Then notice that $p \le s_H$. Hence for $n = p$, we get 
\[
\frac{1}{s_H}\left( 1 + \frac{s_H - 1}{|H : Z(H, G)|}\right) \leq \frac{1}{p}\left( 1 + \frac{p - 1}{|H : Z(H, G)|}\right).
\]
{\it Which shows that the upper bound obtained in Theorem \ref{ThmUB} is an improvement on the known upper bounds for $\Pr(H, G)$}.

\vspace{.2in}

Putting $H = G$ in Theorem \ref{ThmUB}, we get the following result.

\begin{Cor}
Let $G$ be a finite non-abelian group. Then
\[
\frac{1}{l_G}\left( 1 + \frac{l_G - 1}{|G : Z(G)|}\right) \leq \Pr(G) \leq \frac{1}{s_G}\left( 1 + \frac{s_G - 1}{|G : Z(G)|}\right)
\]
with equality if and only if $\cs(G) = \{1, l_G = s_G\}$.
\end{Cor}

We conclude this section with the analogus bounds in terms of smallest and largest character degrees.
Let $\Irr(G)$ denote the set of  complex irreducible  characters of $G$ and $\cd(G) = \{\chi(1) : \chi \in \Irr(G)\}$. In the last decades a number of results have been proved concerning $\cd(G)$. Let $d$ and $m$, respectively denote the smallest and largest degree of  non-linear complex irreducible  characters of $G$.
There are some bounds for $\Pr(G)$ given in \cite[Lemma 2(vi)]{GR06} in terms of the smallest degree of  non-linear complex irreducible  characters of $G$. Precisely, if $G$ is non-abelian then 
\[
\frac{1}{|G'|} < \Pr(G) \leq \frac{1}{|G'|}\left( 1 + \frac{|G'| - 1}{d^2}\right).
\]
The equality holds in the right most inequality if and only if $\cd(G) = \{1, d = m\}$. The lower bound here can be improved by replacing $|G:Z(G)|$ by $m^2$ in the proof of \cite[Theorem 1]{ND10}, which is given in the following equation.

\begin{equation}\label{degbd}
\frac{1}{|G'|}\left( 1 + \frac{|G'| - 1}{m^2}\right) \leq \Pr(G).
\end{equation}
The equality holds in \eqref{degbd} if and only if $\cd(G) = \{1, m = d\}$. Since $\chi(1)^2 \leq |G : Z(G)|$ for all $\chi \in \Irr(G)$ we have the lower bound obtained in \eqref{degbd} is better than \eqref{lbd}.

\section{Some observations on class sizes}

The study of the structure of a finite group $G$ by imposing conditions on the set $\cs(G)$, the set of conjugacy class sizes of $G$, has been studied by many authors in the litrature. Recently A. Camina and R. D. Camina \cite{Camina} gave an impressive survey on this topic. The following questions, which we state in our language, were posed in Section 3.2 of \cite{Camina}:

\noindent {\bf Question.} Let $G$ be a finite group.
\begin{enumerate}
 \item[1.] If all  conjugacy class sizes of $G$, including their multiplicities are known, then what can be said about the solvability of $G$?
 \item[2.] If all  conjugacy class sizes of $G$  are known, then what can be said about the solvability of $G$?
\end{enumerate}

Let $A_n$ denote the alternating group of degree $n$. In 2006, R. M. Guralnick and G. R. Robinson \cite{GR06} proved the following  result.

\begin{Thm}[Theorem 11]\label{ThmGR}
Let $G$ be a finite group such that $\Pr(G) > 3/40$. Then either $G$ is solvable, or else $G \cong A_5 \times T$ for some abelian group $T$, in which case $\Pr(G) = 1/12$.
\end{Thm}

Again in 2006, F. Barry et. al \cite{bM06} proved the following result.

\begin{Thm}[Theorem 4.10]\label{ThmBM1}
Let $G$ be a finite group such that $\Pr(G) > 1/3$. Then $G$ is supersolvable.
\end{Thm}

\begin{Thm}[Theorem 4.12]\label{ThmBM2}
Let $G$ be a finite group of odd order such that $\Pr(G) > 11/75$. Then $G$ is supersolvable.
\end{Thm}

As a consequence of the preceding theorems, we observe the following information regarding above questions:

\noindent{\bf Regarding Question 1.}
Suppose that all conjugacy class sizes of $G$, including their multiplicities are known. Then $\Pr(G)$ can be computed in variuos ways, e.g., by putting $g = 1$ and $H = G$ in Lemma \ref{eq2}. Then Theorems \ref{ThmGR}, \ref{ThmBM1} and \ref{ThmBM2} provide interesting information on the solvability and supersolvability of the group $G$.

\noindent{\bf Regarding Question 2.}  Following N. Ito \cite{nI53}, we define conjugate type vector of a finite groups $G$ by $(1, n_1, n_2, \dots, n_r)$, where $1 < n_1 < n_2 \cdots < n_r$. Also the conjugate rank of $G$, which is denoted by $\crk(G)$, is given by $r$. By putting $g = 1$ and $H = G$ in Lemma \ref{eq2}, we get 

\begin{align}\label{eqn3}
\Pr(G) &= \frac{1}{|G|} \underset{x \in G}{\sum} \frac{1}{|\Cl_G(x)|}\\  
 & \ge  \frac{1}{|G|}\left(|Z(G)| + \crk(G) + \frac{|G| - |Z(G)| - \sum_{i=1}^{r} n_i}{n_r}\right).\nonumber
\end{align}

Given the center and the conjugate type vector of a group $G$ of given order, one can compute $\frac{1}{|G|}(|Z(G)| + \crk(G) + \frac{|G| - |Z(G)| - \sum_{i=1}^{r} n_i}{n_r})$. Hence using \eqref{eqn3}, Theorems \ref{ThmGR}, \ref{ThmBM1} and \ref{ThmBM2} provide interesting information on the solvability and supersolvability of the group $G$.

\noindent\textbf{Acknowledgment.}
The authors are grateful to Dr. A. K. Das for some useful comments on an earlier version of this paper.

\end{document}